\documentclass[10pt]{paper}%
\usepackage[T1]{fontenc}
\usepackage{amsmath,amssymb}
\usepackage{fullpage}
\usepackage{amsthm}
\usepackage{multirow}
\usepackage{color}
\usepackage{pifont}
\usepackage{amsmath}
\usepackage{graphicx}
\usepackage{url}
  \usepackage[all]{xy}
\usepackage{amsfonts}
\usepackage{amssymb}
\usepackage{ytableau}%
\usepackage{showlabels}

\providecommand{\U}[1]{\protect\rule{.1in}{.1in}}
\providecommand{\U}[1]{\protect\rule{.1in}{.1in}}
\providecommand{\U}[1]{\protect\rule{.1in}{.1in}}
\providecommand{\U}[1]{\protect\rule{.1in}{.1in}}
\providecommand{\U}[1]{\protect\rule{.1in}{.1in}}
\newcommand{\ulambda}{{\boldsymbol{\lambda}}}

\newcommand{\umu}{{\boldsymbol{\mu}}}

\newcommand{\Uglov}[2]{{\Phi}^{#1}_{#2}}
\newcommand{\uemptyset }{{\boldsymbol{\emptyset}}}
\newtheorem{Th}{Theorem}[subsection]
\newtheorem{lemma}[Th]{Lemma}
\newtheorem{Cor}[Th]{Corollary}
\newtheorem{Prop}[Th]{Proposition}

\theoremstyle{remark}
\newtheorem{Rem}[Th]{Remark}{\rmfamily}
\theoremstyle{definition}
{\rmfamily}
\newtheorem{exa}[Th]{Example}{\rmfamily}
\newtheorem{abs}[Th]{\bfseries}

\newcommand\blfootnote[1]{%
  \begingroup
  \addtocounter{footnote}{-1}%
  \endgroup
}
\begin{document}

\title{Uglov bipartitions and extended Young diagrams}
\author{Nicolas Jacon}
\maketitle
\date{}
\blfootnote{\textup{2010} \textit{Mathematics Subject Classification}: \textup{20C08,05E10,17B37}} 
\begin{abstract}
We study the class of Uglov bipartitions and prove a generalization of a conjecture by Dipper, James and Murphy. 
 We give two  consequences concerning the computation of canonical bases in affine type $A$ and the description of decomposition matrices for Hecke algebras of type $B_n$ in arbitrary characteristic. 

\end{abstract}

\section{Introduction}

Uglov bipartitions are a class of combinatorial objects which have first be defined  in the context of the representation theory of quantum groups. More precisely, in the case of the affine special linear group $\widehat{\mathfrak{sl}}_e$, these pairs of integer partitions  are known to 
naturally label the crystal graph of the irreducible highest weight modules  (of level two) and thus their canonical bases.   Since their introduction by Uglov in \cite{Ug},
 they have appeared  in various  (but connected) situations: 
 \begin{itemize} 
 \item the representation theory of Cherednik algebras \cite{Bo,CGG} (as the bipartitions indexing the standard modules which are not killed by the KZ functor in type $B_n$), 
 \item the representation theory of Hecke algebras \cite{GJ} (as the bipartitions labelling the so called canonical basic sets in characteristic $0$),
\item  the Harish-Chandra theory for unitary groups \cite{GHJ} (as the bipartitions labelling certain weak Harish-Chandra series).
\end{itemize}

The definition of these bipartitions depends on the choice of a pair of integers ${\bf s}=(s_1,s_2)\in \mathbb{Z}^2$. If ${\bf s}\in \mathbb{Z}^2$ and ${\bf s}'\in \mathbb{Z}^2$
 are in the same orbit modulo an action of an extended affine symmetric group, the associated classes of bipartitions are in bijection. For a certain choice of 
${\bf s}\in \mathbb{Z}^2$, the associated bipartitions are known as FLOTW bipartitions (a special case of ``cylindric bipartitions''), for another choice (called asymptotic),  the associated bipartitions are known as Kleshchev bipartitions. Both types of bipartitions have been extensively studied in recent years. In the general case, 
even if these bipartitions have a nice and relatively easy recursive definition (see \S \ref{cb}), it can be difficult to 
 characterize them explicitly or to study their properties. In \cite{J}, a new combinatorial (but still recursive) definition has been given (it concerns in fact the more general class of Uglov multipartitions).
  As a consequence, an old conjecture by Dipper, James and Murphy has been deduced but only for the class of Kleshchev multipartitions (the papers \cite{H1,H2} consider special cases of this conjecture).  
  This result shows that the Kleshchev multipartitions may be easily obtained as the maximal elements with respect to the  lexicographic order on bipartitions in certain combinatorial expressions.

 The aim of this note is to continue the work of \cite{J} and to obtain a general proof of the Dipper-James-Murphy conjecture for the whole class of Uglov bipartitions (see Theorem \ref{djmc} or its reformulation in Corollary \ref{djmco}). To do this, we use the fact that 
 this conjecture may be easily proved for the class of FLOTW bipartitions   and we use the bijections between classes of Uglov bipartitions that we have already studied and defined in a number of papers.   We in particular obtain a different proof of the conjecture for the class of Kleshchev bipartitions than the one presented in \cite{AJ}.  We note that, even if some of the results of \cite{J} are used in this paper, our result is essentially independant of \cite{J}. In particular, it does not use the notion of staggered sequence which is used 
  to prove the conjecture for the Kleshchev multipartitions in this paper. 
 One of the main interest of this conjecture lies  in its application. We here give two consequences of this result. The first one (which has been already mentioned in \cite{J}) concerns the computation of canonical bases for irreducible highest weight modules fo  $\widehat{\mathfrak{sl}}_e$. The second one concerns the form of the decomposition matrices for Hecke algebras of type $B_n$ in arbitrary characteristic. Thanks to our main result,  we  give an elementary proof for the existence of canonical basic sets for these algebras. Recently, 
  such a result has been also obtained by C. Bowman using the theory of Cherednik algebras \cite{Bo} (even in the wider context of Ariki-Koike algebras).  In general, this  was previously only known  assuming the validity of certain Lusztig's conjectures  on Hecke algebras with unequal parameters (see \cite{GJb}).

\section{Several definitions}
In this part, we recall several combinatorial notions concerning Young diagrams, most of them can be already found in \cite{J}. 

\subsection{Extended Nodes of a bipartition}
 A {\it partition} $\lambda$ of rank $n\in \mathbb{Z}_{>0}$ is a sequence  of non negative and non increasing integers 
 $\lambda=(\lambda_1,\ldots,\lambda_r)$ such that  we have $|\lambda|:=\sum_{1\leq i\leq r} \lambda_i =n$. The unique partition of rank $0$, the empty partition, 
  is denoted by $\emptyset$. We denote by $\ell (\lambda)$ the minimal integer  
such that $\lambda_{\ell (\lambda)}=0$ with the convention that $\ell(\emptyset)=0$. 
 We will be interested here on the set of bipartitions of rank $n$:
  $$\mathcal{P}^2 (n):=\{ (\lambda^1,\lambda^2)\ |\    | \lambda^1 |+| \lambda^2 |=n\}.$$
 The {\it nodes} of the bipartition $\ulambda$ are the elements 
$(a,b,c)$ where $c\in  \{1,2\}$, $a\in \{1,\ldots,\ell(\lambda^c)\}$, $b\in \{1,\ldots,\lambda_a^c\}$. The set of all nodes of $\ulambda$ is  denoted by 
$\mathcal{Y} (\ulambda)$. It is called the {\it Young diagram} of $\ulambda$. 
 The {\it extended nodes} of the bipartition $\ulambda$ are the 
 following elements of $\mathbb{Z}_{\geq 0} \times \mathbb{Z}_{\geq 0}\times \{1,2\}$:
\begin{enumerate}
\item the elements of $\mathcal{Y} (\ulambda)$,
\item the elements of the form $(0,b,c)$ where $b>\lambda^c_1$ and $c\in  \{1,2\}$,
\item the elements of the form  $(a,0,c)$ where $a>\ell (\lambda^c)$ and $c\in \{1,2\}$.
\end{enumerate}
The set of all extended nodes of $\ulambda$ is called the {\it extended Young diagram} of $\ulambda$ and it is denoted by $\mathcal{Y}^{\textrm{ext}} (\ulambda)$. It thus contains the Young diagram of $\ulambda$. One can represent it as a collection  of boxes. Each box then corresponds to an extended node of the bipartition as in the following example. 
  \begin{exa}
  We consider the bipartition $(3.3.1,2.1)$ of $n=10$. The  extended Young diagram is given as follows. 
  
\vspace{1cm}  
  
  \centerline{
 \Bigg(\  \ \ \ \ytableausetup
{mathmode}
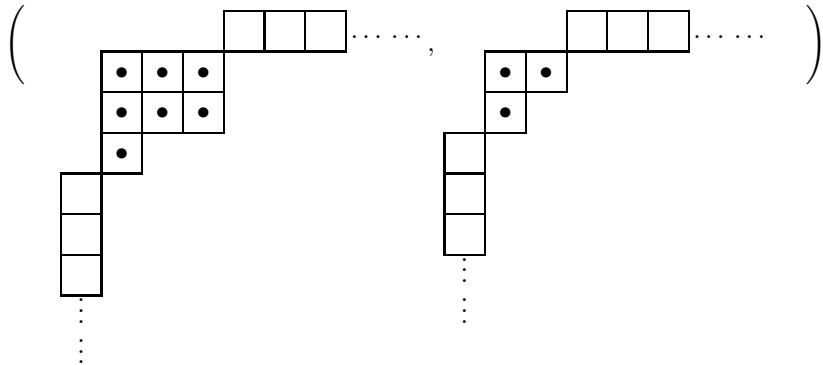
\begin{ytableau}
\none  &  \none  & \none  & \none  & \ & \ &\ & \none[\dots]& \none[\dots] \\
 \none & \bullet  & \bullet  & \bullet \\
\none & \bullet & \bullet & \bullet \\
\none & \bullet \\  
\  \\
\  \\
\  \\
\none[\vdots] \\
 \none[\vdots] \\
\end{ytableau},
\begin{ytableau}
\none  &  \none  & \none   & \ & \ &\ & \none[\dots]& \none[\dots] \\
\none  &  \bullet & \bullet  \\
\none & \bullet \\  
\  \\
\  \\
\  \\
\none[\vdots] \\
 \none[\vdots] \\
\end{ytableau}\  \ \ \  \Bigg)}
 The boxes containing a bullet correspond to the boxes of the usual Young diagram.
 \end{exa}

Let  $e\in \mathbb{Z}_{>1}\sqcup\{\infty\}$.  We now fix  ${\bf s}=(s_1,s_2)\in \mathbb{Z}^2$.  One can attach to each extended node $\gamma=(a,b,c)$ of the extended Young diagram its {\it content} (depending  on the choice of ${\bf s}$):  
$$\operatorname{cont} (\gamma)=b-a+s_c\in \mathbb{Z}.$$
By definition, the {\it residue} (depending  on the choice of ${\bf s}$ and $e$) $\operatorname{res} (\gamma)$ of the extended node $\gamma$ is the content modulo $e$ if $e$ is finite and the content if $e=\infty$. Throughout the paper, we set $I:= \mathbb{Z}/e\mathbb{Z}$ (which will be identified with $\{0,\ldots,e-1\}$) if $e$ is finite 
 and $I:=\mathbb{Z}$ if $e=\infty$. 
If $\operatorname{res} (\gamma)=\mathfrak{j}$ then we say that $\gamma$ is a {\it (extended) $\mathfrak{j}$-node.}

The {\it boundary} of the extended Young diagram is by definition given by the extended nodes  $(a,b,c)\in \mathcal{Y}^{\text{ext}} (\ulambda)$ such that:
\begin{itemize}
\item  $(a,b+1,c)$ is not in $\mathcal{Y}^{\textrm{ext}} (\ulambda)$.  Such nodes constitute the  {\it vertical boundary} of the extended Young diagram.
\item  $(a+1,b,c)$ is not in $\mathcal{Y}^{\textrm{ext}} (\ulambda)$.  Such nodes constitute the  {\it horizontal boundary} of the extended Young diagram.
\end{itemize}
The extended nodes which are in the vertical or horizontal boundary without being in the Young diagram are called {\it virtual}.

A node $\gamma=(a,b,c)$ of $\mathcal{Y} (\ulambda)$ is said to be {\it removable} for $\ulambda$  if  
 $\mathcal{Y} (\ulambda)\setminus \{\gamma\}$ is the Young tableau 
  of a bipartition $\umu$.  If $\gamma=(a,b,c)\in \mathbb{Z}_{> 0} \times \mathbb{Z}_{> 0}\times \{1,\ldots, l\}$ is such that $\mathcal{Y} (\ulambda)\sqcup \{\gamma\}$ is the Young tableau of 
   a bipartition $\umu$ then it is said to be {\it addable} for $\ulambda$.  
  
The intersection between the vertical and the horizontal boundary is given by the set of  removable nodes. We see also that  there always exists one unique extended node in a fixed component with a given content which is either addable, either in the boundary of $\ulambda$ . We will denote by $\mathcal{E}_{\mathfrak{j}} (\ulambda)$ the set consisting of:
\begin{itemize}
\item addable $\mathfrak{j}$-nodes of $\ulambda$, we say that these nodes are  of {\it nature} $A$. 
\item  extended $\mathfrak{j}$-nodes of the boundary of  $\ulambda$. Such node may be 
 either removable (we then say that they are of nature $R$), either in the vertical boundary without being removable (we say that they are  of nature $B_v$) or in the horizontal boundary without being removable (of nature $B_h$)
\end{itemize}

Given the nature of a node of content $j$ in a component $c\in \{1,2\}$ of a bipartition, there is always only two possibilities for the nature 
 of the nodes with content $j-1$ and $j+1$ in the same component,  in respectively  $\mathcal{E}_{\mathfrak{j}-1} (\ulambda)$ and $\mathcal{E}_{\mathfrak{j}+1} (\ulambda)$. They are given in the following table:
 
             $$\begin{array}{|c||c||c|}
      \hline
       \text{Possible nature }& \text{Nature }  &  \text{Possible nature}   \\
        \text{of the node of content $j-1$} & \text{of the node of content $j$} &  \text{of the node of content $j+1$}  \\
      \hline
   B_v\text{ or }R   & A & B_h\text{ or }R    \\
        \hline 
    B_h\text{ or }A    & R & B_v\text{ or }A       \\
        \hline 
     B_v\text{ or }R           & B_v & B_v\text{ or }A       \\
        \hline 
      B_h\text{ or }A        & B_h & B_h\text{ or }R       \\
        \hline 
     \end{array}$$

  \begin{exa}
  We consider the bipartition $\ulambda=(3.3.1,2.1)$ of $n=10$ and ${\bf s}=(0,1)$, $e=3$.  Here is the associated extended Young diagram with the residue 
   of each node written in the associated box:
   \vspace{1cm}  
  
  \centerline{
 \Bigg(\  \ \ \ \ytableausetup
{mathmode}
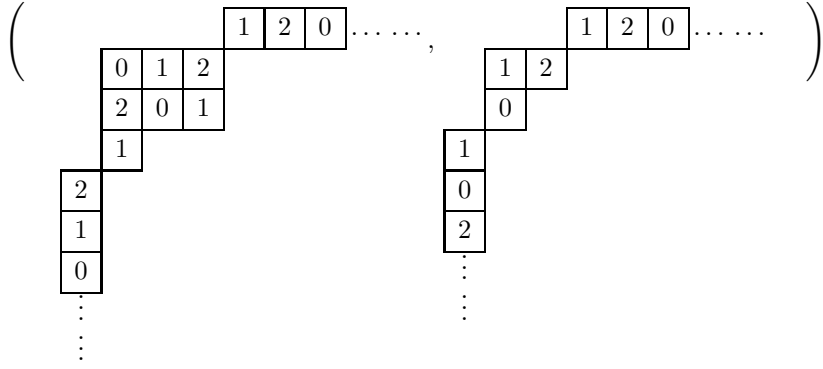
\begin{ytableau}
\none  &  \none  & \none  & \none  & 1 & 2 &0 & \none[\dots]& \none[\dots] \\
 \none &0 & 1  & 2 \\
\none & 2 & 0 & 1 \\
\none & 1 \\  
2  \\
1  \\
0  \\
\none[\vdots] \\
 \none[\vdots] \\
\end{ytableau},
\begin{ytableau}
\none  &  \none  & \none   & 1 & 2 &0 & \none[\dots]& \none[\dots] \\
\none  &  1 & 2  \\
\none & 0 \\  
1  \\
0 \\
2  \\
\none[\vdots] \\
 \none[\vdots] \\
\end{ytableau}\  \ \ \  \Bigg)}
$\mathcal{E}_{2} (\ulambda)$ consists in  the following extended nodes: 
\begin{itemize}
\item the nodes $(0,5+3k,1)$ with  $k\in \mathbb{N}$ which are of nature $B_h$. 
\item the nodes $(0,4+3k,2)$ with  $k\in \mathbb{N}$ which are of nature $B_h$. 
\item the nodes $(4+3k,0,1)$ with  $k\in \mathbb{N}$ which are of nature $B_v$. 
\item the nodes $(5+3k,0,2)$ with  $k\in \mathbb{N}$ which are of nature $B_v$. 
\item $(1,3,1)$ which is of nature $B_v$, $(3,2,1)$ which is of nature $A$,  $(1,2,2)$ which is of nature $R$
 and $(3,1,2)$ which is of nature $A$. 
\end{itemize}
 \end{exa}

\subsection{Order on nodes of a bipartition}
Let $\ulambda$ be a bipartition and $\gamma_1$ (resp. $\gamma_2$) be an extended  node of $\ulambda$   or an addable node  of $\ulambda$. We write $\gamma_1 <_{\bf s}  \gamma_2$
 if and only if 
   \begin{itemize}
  \item $\operatorname{cont}(\gamma_1)<\operatorname{cont}(\gamma_2)$ or,
\item $\operatorname{cont}(\gamma_1)=\operatorname{cont}(\gamma_2)$, $c_1=2$ and $c_2=1$.

  \end{itemize}
Let now $\mathfrak{j} \in I$.  Then the nodes in
$\mathcal{E}_{\mathfrak{j}} (\ulambda)$ are all comparable and  
we see that we have  $\gamma_1 <_{\bf s}\gamma_2$  for such two nodes and that these two nodes are consecutive if we are in one of the following two cases:
\begin{itemize}
\item $\operatorname{cont}(\gamma_1)=\operatorname{cont}(\gamma_2)$, $c_1=2$ , $c_2=1$,
\item $\operatorname{cont}(\gamma_1)+e=\operatorname{cont}(\gamma_2)$,  $c_1=1$, $c_2=2$,  (if $e$ is finite, if $e=\infty$ only the above case occurs because we only have two nodes with a given content). 

\end{itemize}


%

\subsection{Order on bipartitions}\label{ord}

  We now define a total  order $\preceq_{\bf s}$ on the set of bipartitions depending on ${\bf s}$. To do this, let $\ulambda \in \mathcal{P}^2 (n)$, consider the set of all the nodes on the vertical boundary of $\ulambda$ and write them in decreasing order with respect to $<_{\bf s}$ (such nodes are always comparable because there is no nodes in the same component with the same content which are in the vertical boundary). We obtain an infinite sequence $(\gamma_{i} (\ulambda))_{i\in  \mathbb{Z}_{>0}}$ which will be called the {\it boundary sequence}. We write $\ulambda  \preceq_{\bf s} \umu$ if and only $\ulambda=\umu$ or if there exists $j>0$ such that 
$$\forall 0<i<j,\ \gamma_i (\ulambda)=\gamma_i (\umu),\ \text{and }  
\gamma_j (\ulambda)  <_{\bf s}  \gamma_j (\umu).$$
Consider the lexicographic order on the set of bipartitions, that is, $\ulambda  \leq  \umu$
 if and only $\lambda=\mu$ or if there exists $j>0$ such that 
$$\forall 0<i<j,\ \lambda^1_i =\mu^1_i,\ \text{and }  \lambda^1_j  <  \mu^1_j$$
or $\lambda^1=\mu^1$  and  there exists $j>0$ such that 
$$\forall 0<i<j,\ \lambda^2_i =\mu^2_i,\ \text{and }  \lambda^2_j  <  \mu^2_j.$$
We write  $\ulambda  \prec_{\bf s} \umu$  if  $\ulambda  \preceq_{\bf s} \umu$  and  $\ulambda  \neq  \umu$. We have the following particular case.
\begin{Prop}\label{djmasy}
Assume that ${\bf s}=(s_1,s_2)$ is such that  $s_1-s_2>n-1$ then we have 
$$\ulambda  \preceq_{\bf s} \umu \iff \ulambda  \leq  \umu. $$
\end{Prop}
\begin{proof}
Assume that  $s_1-s_2>n-1$ then if $\gamma=(a,b,c)$ and $\eta=(a',b',c')$ are two nodes on the vertical boundary of $\ulambda$ we have that 
 $\gamma \leq_{\bf s} \eta$ if and only if:
 \begin{itemize}
 \item $c'=1$ and $c=2$,
 \item $c=c'$ and $a<a'$.
 \end{itemize}
 This implies that the order $\preceq_{\bf s}$ is the same as the usual lexicographic order on bipartitions. 
\end{proof}
\begin{Rem}
Of course, the order $\preceq_{\bf s}$  strongly depends on the choice of  ${\bf s}$ and it does not correspond to the lexicographic order in general if $s_1-s_2\leq n-1$.

\end{Rem}

\subsection{Boundary sequence  and nature of nodes}\label{Nord}

We study in details the relations between the boundary sequence and the nature of some nodes of a bipartition. Let $\ulambda=(\lambda^1,\lambda^2)$ 
 be a bipartition and ${\bf s}=(s_1,s_2)\in \mathbb{Z}^2$.   For each $j\in \mathbb{Z}$ and each component, there is a unique extended node with content $j$
 which is in this component and which is either addable either in the boundary of $\ulambda$.

Also note that this node is of nature $B_{v}$ or $R$ if and only if $j$ appears in the boundary sequence. 
 As a consequence, one may easily obtain 
 the boundary sequence from a table, called the {\it table of natures}, listing all the nature of the extended nodes associated with each $j\in \mathbb{Z}$ (representing the content of the node) 
  and $c\in \{1,2\}$ (representing the component of the node) as in the following examples.

     \begin{exa}
Let   ${\bf s}=(0,1)$ and  consider the bipartition $\ulambda=(6.1,2.2)$.          
      \vspace{0,3cm}
      
  \centerline{
 \Bigg(\  \ \ \ \ytableausetup
{mathmode}
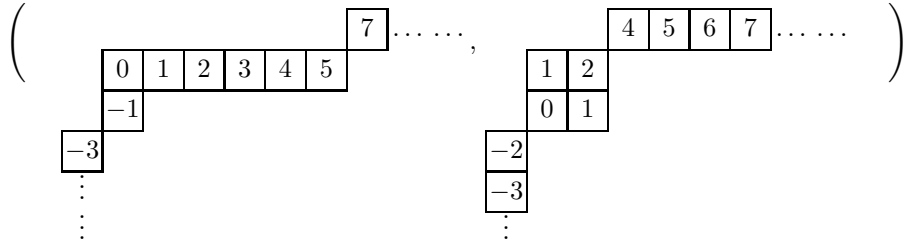
\begin{ytableau}
\none  &  \none  & \none  & \none  & \none & \none & \none &7 & \none[\dots]& \none[\dots] \\
 \none &0 & 1  & 2 & 3 & 4 &  5\\
\none & -1 \\  
-3  \\
\none[\vdots] \\
 \none[\vdots] \\
\end{ytableau},
\begin{ytableau}
\none  &  \none  & \none   & 4 &5&6&7  & \none[\dots]& \none[\dots] \\
\none  &  1 & 2  \\
\none & 0  & 1\\  
-2  \\
-3 \\
 \none[\vdots] \\
\end{ytableau}\  \ \ \  \Bigg) 
}

   The table of natures  gives the nature of all the (extended) nodes which are either addable either in the boundary, of content between $-3$ and $6$ written in increasing order with respect to ${(0,1)} $ for $\ulambda$. The nodes of content 
    greater than $6$ are all   virtual nodes of the horizontal boundary and the nodes of content 
   less  than $-3$ are all   virtual nodes of the vertival  boundary

            \vspace{0,3cm}
      
      {\small
      $$\begin{array}{|c||c|c|c|c|c|c|c|c|c|c|c|c|c|c|c|c|c|c|c|c|c|c|}
      \hline
      \text{Component}  & 2 & 1 & 2 &  1 & 2 & 1 & 2 &  1 & 2 & 1 & 2 & 1 & 2 & 1 & 2 &  1 & 2 & 1 & 2 &  1 \\
      \hline
        \text{Content} & -3 & -3 & -2 & -2 &  -1 & -1 & 0 & 0 &  1 & 1 & 2 & 2 & 3 & 3 & 4 & 4 &  5 & 5 & 6 & 6  \\
        \hline 
         (6.1,2.2) & {\bf B_v} & {\bf B_v} & {\bf B_v} & A &  A & {\bf R} & B_h& A &  {\bf R} & B_h & {\bf B_v} & B_h & A & B_h & B_h & B_h &  B_h & {\bf R} & B_h & A   \\
          \hline 
     \end{array}$$}

     From this, we can write the infinite sequence of \S \ref{ord}  which allows to compare  bipartitions with respect to $\preceq_{{\bf s}}$.   
     If a node if of nature $B_v$ or $R$, then this means that this node is in this sequence otherwise, it is not. For example,  
     take the bipartition $(6.3,2.1)$ then with respect to $(0,1)$ we have:
     
           {\small
      $$\begin{array}{|c||c|c|c|c|c|c|c|c|c|c|c|c|c|c|c|c|c|c|c|c|c|c|}
      \hline
      \text{Component}  & 2 & 1 & 2 &  1 & 2 & 1 & 2 &  1 & 2 & 1 & 2 & 1 & 2 & 1 & 2 &  1 & 2 & 1 & 2 &  1 \\
      \hline
        \text{Content} & -3 & -3 & -2 & -2 &  -1 & -1 & 0 & 0 &  1 & 1 & 2 & 2 & 3 & 3 & 4 & 4 &  5 & 5 & 6 & 6  \\
        \hline 
         (6.3,2) & {\bf B_v} & {\bf B_v} & {\bf B_v} & A &  {\bf B_v} & B_h & A& B_h &  {B}_h & {\bf R} & {\bf R} &A & A & B_h & B_h & B_h &B_h   &  {\bf R} & B_h & A   \\
          \hline 
     \end{array}$$}
     and we immediately see that $(6.1,2.2)\prec_{(0,1)} (6.3,2)$. 
     
     \end{exa}

     In $\umu$ is another bipartition, then we say that the two extended or addable nodes of $\ulambda$  and  $\umu$ with the same content and in the same component 
    has the same {\it general nature} if the nature of these nodes are both  in $\{R,B_v\}$ or both in $\{A,B_h\}$.  Hence, one can compare two bipartitions 
        with respect to $\prec_{\bf s}$ by looking at the table of natures of $\umu$ and $\ulambda$   and  by checking  when, starting from the left, the general nature of extended or addable nodes differs for 
         $\ulambda$ and $\umu$.  In the above case, this happens for the content $1$ and component $1$.

\section{Uglov bipartitions}

In this section, we define the notion of Uglov bipartitions, recall the context in which they appear  and recall several properties. 

\subsection{Definition}

Let $\mathfrak{g}_e$ be the quantum affine algebra of type $A^{(1)}_{e-1}$. Let ${\bf s} \in \mathbb{Z}^2$.  
 For each $n\in \mathbb{Z}_{\geq 0}$, we have a $\mathbb{C}$-vector space:
 $$\mathcal{F}_n= \bigoplus_{\ulambda \in \mathcal{P}^2 (n)}  \mathbb{C}   \ulambda.$$
The Fock space of level two is then the following $\mathbb{C}$-vector space:
$$\mathcal{F}=\bigoplus_{n\in \mathbb{Z}_{\geq 0}} \mathcal{F}_n. $$
There is an action of $\mathfrak{g}$  on $\mathcal{F}$ which gives a structure of an  integrable module.  
In particular, the action of the Chevalley operators $e_{\mathfrak{j}}$ and  $f_{\mathfrak{j}}$ (with $\mathfrak{j}\in I$)    on the Fock space are given by:
 $$\forall \mathfrak{j}\in I,\forall \ulambda \in \mathcal{P}^2 (n),\ f_\mathfrak{j} \ulambda=\sum_{\mathcal{Y}(\umu)=\mathcal{Y}(\ulambda)\sqcup\{\gamma\},\ \operatorname{res}(\gamma)=\mathfrak{j}} \umu,  e_\mathfrak{j}\ulambda=\sum_{\mathcal{Y}(\ulambda)=\mathcal{Y}(\umu)\sqcup\{\gamma\},\ \operatorname{res}(\gamma)=\mathfrak{j}} \umu.$$
 The submodule $V_{\bf s}$ generated by the empty bipartition is then an irreducible highest weight module with weight $\Lambda_{s_1}+\Lambda_{s_2}$ (where the $\Lambda_i$'s, $i\in I$, denote the fundamental weights).

\subsection{Canonical bases}\label{cb}

 Let ${\boldsymbol{\lambda}}$ be a bipartition,  
and $\mathfrak{j}\in I$. We consider the set of addable
and removable $\mathfrak{j}$-nodes of ${\boldsymbol{\lambda}}$. 
We define a word obtained by reading  these nodes in the increasing order with respect to $\prec_{{\mathbf{s}}}$, 
If a removable $\mathfrak{j}$-node appears just before an addable $\mathfrak{j}$-node, we delete both 
and continue the same procedure as many times as possible. In the end, we reach a word 
 of nodes such that the first $p$ nodes are addable and the last $q$ nodes are removable, for some $p, q\in\mathbb{N}$. 
If $p>0,$ let $\gamma$ be
the rightmost addable $\mathfrak{j}$-node. If it exists, the node $\gamma$ is called the \emph{good} $\mathfrak{j}$-node of $\boldsymbol{\lambda}$. 

By definition, $\ulambda$ is said to be an {\it Uglov bipartition} of rank $n>0$ if there exists a sequence of bipartitions 
$\ulambda^{[1]}:=(\emptyset,\emptyset)$, $\ulambda^{[2]}$, \ldots, $\ulambda^{[n]}:=\ulambda$ such that for each $j\in \{1,\ldots,n\}$, the bipartition $\ulambda^{[j]}$ is in $\mathcal{P}^2 (j)$ and such that for each $j\in \{2,\ldots,n\}$
$\lambda^{[j]}$ is obtained by
adding a good  $\mathfrak{j}$-node to $\lambda^{[j-1]} $.

The set of Uglov bipartitons of rank $n$ is denoted by  $\Uglov{e}{\bf s}(n)$ and the set of all Uglov bipartitions by  
$\Uglov{e}{\bf s}$. Note that $(\emptyset,\emptyset)$ is the unique Uglov bipartition with rank $0$.

Let us come back to the submodule $V_{\bf s}$  of the Fock space. By \cite[Th 6.6.14]{GJ} (this is a result by Uglov using Kashiwara-Lusztig theory of canonical bases for quantum groups), there exists a basis, called the canonical basis for $V_{\bf s}$  (as a $\mathbb{C}$-vector space) which is labeled by the set of Uglov bipartitions:
$$\{ G(\ulambda,{\bf s})\ |\ \ulambda \in \Uglov{e}{\bf s}\}.$$
In fact, we even get that the set 
$$\{ G(\ulambda,{\bf s})\ |\ \ulambda \in \Uglov{e}{\bf s}(n)\}$$
is a basis of $V_{\bf s}\cap \mathcal{F}_n $.

 This basis enjoys nice properties which we not recall here (see \cite[Th. 6;6;11]{GJ}). All we need in the folllowing is:
 \begin{Prop}[Uglov \cite{Ug}]\label{max}
   For each 
$\ulambda \in \Uglov{e}{\bf s} (n)$, we have:
$$G(\ulambda,{\bf s})=\ulambda+\text{Linear combination of $\umu$ of rank $n$ with }\umu \prec_{\bf s} \ulambda.$$ 
\end{Prop}

\begin{lemma}
For all sequences of residues $(\mathfrak{j}_1,\ldots,\mathfrak{j}_n)\in I^n$, the maximal element with respect to $\prec_{\bf s}$ in $f_{\mathfrak{j}_1}\ldots f_{\mathfrak{j}_n}.\uemptyset$ is an Uglov bipartition of rank $n$. 
\end{lemma} 
\begin{proof}
Each canonical basis element is a linear combination of bipartitions and the maximal one with respect to $\preceq_{\bf s}$ is an Uglov bipartition. As $f_{\mathfrak{j}_1}\ldots f_{\mathfrak{j}_n}.\uemptyset$  is a linear combination of these canonical basis elements, the results follows.

\end{proof}
%

\subsection{Bijections between Uglov bipartitions}\label{bij}

Let $\widehat{\mathfrak{S}}_2$  be the (extended) affine symmetric group. We denote by $P_2:=\mathbb{Z}^2$ the $\mathbb{Z}$-module with  standard basis $\{y_1,y_2\}$. We denote by $\sigma_1$ the generator of  $\mathbb{Z}/2\mathbb{Z}$. Then 
$\widehat{\mathfrak{S}}_2$ can be seen as 
  the semi-direct product $P_2 \rtimes \mathbb{Z}/2\mathbb{Z}$ where the relations  are given by  $\sigma_1 y_1 \sigma_1=y_{2}$.   
This group acts faithfully on $\mathbb{Z}^2$  by setting for any ${{{\bf s}}}=(s_{1},s_{2})\in 
\mathbb{Z}^{2}$: %
$$\begin{array}{rcll}
\sigma_1. (s_1,s_2)&=&(s_2,s_1) ,\\
y_1. (s_1,s_2)&=&(s_1+e,s_2) ,\\
y_2. (s_1,s_2)&=&(s_1,s_2+e).
\end{array}$$
 Set $\tau:=\sigma_1 y_1$ then $\widehat{\mathfrak{S}}_2$ is generated by $\sigma_1$ and $\tau$.  Moreover, we have  
a fundamental domain for this action  given by:
$$\mathcal{S}_e:=\left\{(s_1,s_2)\in \mathbb{Z}^2\ |\ 0\leq s_1 \leq  s_2 <e  \right\}.$$
If ${\bf s}$ is in $\mathcal{S}_e$, the set  $\Uglov{e}{{\bf s}}$ has a nice non recursive  definition (see \cite[Def. 5.7.8]{GJ}) but such 
  definition is not available in the general case. Nevertheless, one can use the following method to 
 compute the sets $\Uglov{e}{{\bf s}}$ in the general case.

We know that if ${\bf s}_1$ and ${\bf s}_2$ are in the same class modulo the action of $\widehat{\mathfrak{S}}_2$ then both  modules $V_{{\bf s}_1}$ and $V_{{\bf s}_2}$ are isomorphic and there is a bijection 
 $$\Psi^e_{{\bf s}_1\to {\bf s}_2}:  \Uglov{e}{{\bf s}_1}\to \Uglov{e}{{\bf s}_2}, $$ 
 which enjoys nice properties with respect to the module structure. In particular, this is a crystal isomorphism. This means that if $\ulambda \in \Uglov{e}{{\bf s}_1}$
  and if $\ulambda' \in \Uglov{e}{{\bf s}_1}$  is obtained from $\ulambda$ by removing a good $\mathfrak{j}$-node. Then 
 $\Psi^e_{{\bf s}_1\to {\bf s}_2} (\ulambda)$ admits a removable good $\mathfrak{j}$-node and if we remove it, we obtain a bipartition $\umu'$ satisfying 
 $\Psi^e_{{\bf s}_1\to {\bf s}_2} (\ulambda')=\umu'$. 
 
\begin{itemize}
\item if ${\bf s}_2=\tau. {\bf s}_1$, by \cite[Prop. 3.1(2)]{JU2},  we have 
$$\Psi^e_{{\bf s}_1\to {\bf s}_2} (\lambda^1,\lambda^2)=(\lambda^2,\lambda^1).$$

\item  if ${\bf s}_2=\sigma_1 .{\bf s}_1$, the bijection  has been combinatorially described in 
 \cite{JU2}. We don't need the explicit form of this bijection but only some properties which will be recall once we need them. 
\end{itemize}
Otherwise, the map $\Psi^e_{{\bf s}_1\to {\bf s}_2} $ is a composition of maps of the above form.  If $\ulambda \in   \Uglov{e}{{\bf s}_1}$ then 
 the Uglov bipartitions $\umu \in   \Uglov{e}{\sigma.{\bf s}_1}$ such that $\Psi^e_{{\bf s}_1\to \sigma.{\bf s}_1} (\ulambda)=\umu $ 
 for all $\sigma \in \widehat{\mathfrak{S}}_2$ are said to be the bipartitions in the isomorphism class of $\ulambda$.

\begin{exa}
Take ${\bf s}=(0,1)$ and $e=3$. We consider the bipartition $\ulambda=(6.1,2.2)$ which is in $\Uglov{e}{\bf s}$. Then one can compute 
 $\Psi^e_{{\bf s}\to \sigma.{\bf s}} (\ulambda)$  for $\sigma\in \widehat{\mathfrak{S}}_2$ using the algorithm  in \cite{JU2} or one can use the program given in \cite{Jgap}. We get for all $k\in \mathbb{Z}$:
$$\Psi^e_{(0,1)\to (1+3k,0)} (\ulambda)=
\left\{ \begin{array}{rcl}
(5.2.1,3) & \text{ if } k\geq  0 \\
(2.2,6.1) & \text{ if } k=-1  \\
 (2.1,6.1.1) & \text{ if } k< -1 \\
\end{array}
\right.\ \ 
\Psi^e_{(0,1)\to (0,1+3k)} (\ulambda)=
\left\{ \begin{array}{rcl}
(3,5.2.1) & \text{ if } k> 1 \\
 (6.1.1,2.1) & \text{ if } k< 0
\end{array}
\right.
$$
\end{exa}
 
 Thus one may recover the set of Uglov bipartitions 
 $\Uglov{e}{{\bf s}}$ for all ${\bf s}\in \mathbb{Z}^2$ from the set 
 $\Uglov{e}{{\bf s}'}$ where ${\bf s}'\in \mathcal{S}_e$ is in the orbit of ${\bf s}$ modulo the above action and from the use of the above known bijections.

  Finally, note that, given a bipartition $\ulambda\in \Uglov{e}{{\bf s}}$, the isomorphism $\Psi^e_{{\bf s}\to \sigma.{\bf s}}$
  affects the table of natures for the $\mathfrak{j}$-nodes. In fact, this table may be easily obtained from the one of 
  $\ulambda$ in the case where $\sigma:=\tau$: we just have to translate the natures of the nodes by one box to the right  for the nodes in component $2$. For the nodes in  component $1$,
   we have to translate them by $2e-1$ boxes.  Here is an example for $e=2$
  
        {\small
      $$\begin{array}{|c||c|c|c|c|c|c|c|c|c|c|c|c||}
      \hline
      \text{Component}&  \ldots&  2 &  1 & 2 & 1 & 2 &  1 & 2 & 1 & 2 &1 & \ldots \\
      \hline
        \text{Content} & \ldots &-2 & -2 &  -1 & -1 & 0 & 0 &  1 & 1 & 2 & 2& \ldots  \\
        \hline 
       \ulambda & \ldots & X_1 & X_2& X_3 &  X_4 & X_5 & X_6 & X_7 &  {X_8} & X_9 &.  & \ldots   \\
          \hline 
       \Psi^e_{{\bf s}\to \tau.{\bf s}}(\ulambda) & \ldots & ? & X_1 & ? & X_3 & X_2  & X_5 & X_4   & X_7 &  X_6  & X_9  &\ldots   \\
         \hline 
     \end{array}$$}
  
Such table is less elementary  in the case $\sigma=\sigma_1$  but a table in \cite[\S 6.1.2]{J} explains the different possibilities: for a given content $j\in \mathbb{Z}$ and two
     addable or extended nodes $\mathfrak{j}$-nodes $\gamma_2$ and $\gamma_1$  of content $j$ in component $2$ and $1$, 
   the natures of the two associated nodes in component $1$ and $2$ are transformed into nodes with specific natures:

    \begin{center}
\begin{tabular}{|c|c||c|c|}
   \hline
   \multicolumn{2}{|c||}{{\bf Nodes  in} $\ulambda$} & \multicolumn{2}{c|}{\bf Nodes   in $\Psi^e_{{\bf s}\to \sigma_1.{\bf s}}  (\ulambda)$\ } \\
      \hline
 Component  $2$ &  Component $1$ &     Component  $2$ &  Component $1$  \\
   \hline 
   \hline 
   $R$ &$R$& $R$ &$R$ \\
   \hline
   $A$ &$R$& $A$ & $R$ \\
   \hline
   \multirow{2}*{$B_{v}$} &\multirow{2}*{$R$}& $R$ & $B_{v}$  \\
    \cline{3-4}
    & &    $B_{v}$& $R$ \\
   \hline
   $B_{h}$ &$R$& $B_{h}$ & $R$ \\
   \hline
   \multirow{2}*{$R$} & \multirow{2}*{$A$}& $R$ &$A$ \\
   \cline{3-4}
    & &    $B_{h}$& $B_{v}$ \\
   \hline
   $A$ &$A$& $A$ & $A$ \\
   \hline
   \multirow{2}*{$B_{v}$} &\multirow{2}*{$A$}& $B_{v}$ & $A$ \\
   \cline{3-4}
    & & $A$ & $B_{v}$ \\
   \hline 
   $B_{h}$ &$A$& $B_{h}$ & $A$ \\
   \hline
    \multirow{2}*{$R$} & \multirow{2}*{$B_{h}$}& $R$ &$B_{h}$  \\
    \cline{3-4}
     & & $B_{h}$ & $R$\\
   \hline
    \multirow{2}*{$A$} & \multirow{2}*{$B_{h}$}& $A$ & $B_{h}$ \\
    \cline{3-4}
     & & $B_{h}$ & $A$ \\
   \hline
    \multirow{3}*{$B_{v}$} & \multirow{3}*{$B_{h}$}& $R$ & $A$ \\
    \cline{3-4}
    & & $B_{v}$ & $B_{h}$ \\
     \cline{3-4}
    & & $B_{h}$ & $B_{v}$ \\
   \hline
   $B_{h}$ &$B_{h}$& $B_{h}$ & $B_{h}$ \\
   \hline
   $R$ &$B_{v}$& $R$ &$B_{v}$ \\
   \hline
   $A$ &$B_{v}$& $A$ & $B_{v}$ \\
   \hline
   $B_{v}$ &$B_{v}$& $B_{v}$ & $B_{v}$ \\
   \hline
   $B_{h}$ &$B_{v}$& $B_{h}$ & $B_{v}$ \\
   \hline
\end{tabular} 
\end{center}

   For example if 
 $ \gamma_2$ is of nature $R$ and $\gamma_1$ of nature $B_h$ then the nodes  $\gamma_2'$ and $\gamma_1'$  of content $j$ in component $2$ and $1$
  of   $ \Psi^e_{{\bf s}\to \sigma_1.{\bf s}}(\ulambda) $ may  be of nature $R$ and $B_h$, or $B_h$ and $R$. 
  
  Note however that there is a nice property which is verified by the bijection $\Psi^e_{{\bf s}\to \sigma_1.{\bf s}}$: by \cite{JL}, this bijection does not depend on $e$ and we thus have 
   $\Psi^e_{{\bf s}\to \sigma.{\bf s}}=\Psi^\infty_{{\bf s}\to \sigma.{\bf s}}$. 
  
   \begin{exa}
     Consider the bipartition $\ulambda=(6.1,2.2)$ which is in $\Uglov{3}{(0,1)}$. We have already seen that 
      $\umu:=\Psi^3_{(0,1)\to (1,0)}(5.2.1,3)$.  Write the Young tableau of these two bipartitions with the associated content:
      
      \vspace{0,3cm}
      
  \centerline{
 \Bigg(\  \ \ \ \ytableausetup
{mathmode}
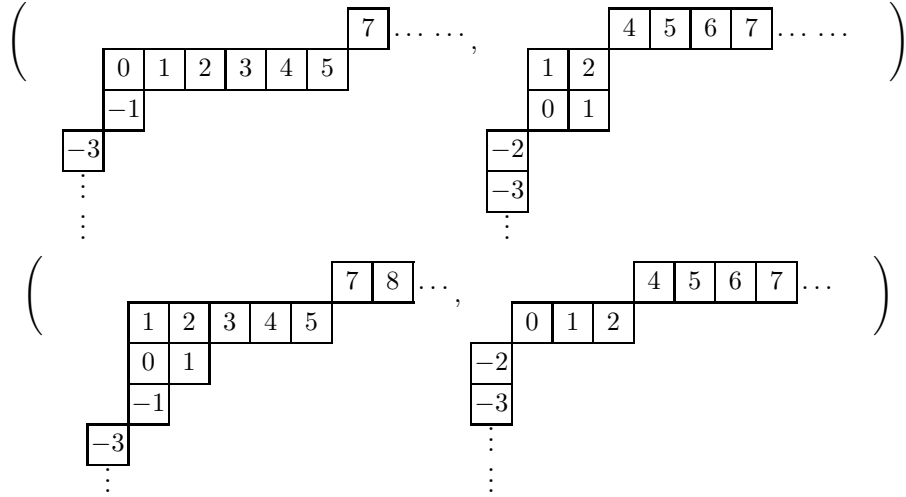
\begin{ytableau}
\none  &  \none  & \none  & \none  & \none & \none & \none &7 & \none[\dots]& \none[\dots] \\
 \none &0 & 1  & 2 & 3 & 4 &  5\\
\none & -1 \\  
-3  \\
\none[\vdots] \\
 \none[\vdots] \\
\end{ytableau},
\begin{ytableau}
\none  &  \none  & \none   & 4 &5&6&7  & \none[\dots]& \none[\dots] \\
\none  &  1 & 2  \\
\none & 0  & 1\\  
-2  \\
-3 \\
 \none[\vdots] \\
\end{ytableau}\  \ \ \  \Bigg) 
}

        \centerline{
 \Bigg(\  \ \ \ \ytableausetup
{mathmode}
\begin{ytableau}
\none  &  \none  & \none   &  \none & \none  &  \none & 7 & 8& \none[\dots] \\
\none  &  1 & {2}  & 3 & 4 & 5\\
\none & 0 & 1  \\  
\none & -1  \\
-3 \\
 \none[\vdots] \\
\end{ytableau},
\begin{ytableau}
\none  &  \none  & \none  & \none  & 4 & 5   & 6&7&  \none[\dots] \\
 \none &0 & 1  & {2} \\
 { -2} \\  
-3 \\
\none[\vdots] \\
 \none[\vdots] \\
\end{ytableau}\  \ \ \  \Bigg) 
}
      
      The following table gives the nature of all the nodes of content between $-3$ and $6$ written in increasing order with respect to ${(0,1)} $ for $\ulambda$ and  ${(1,0)}$ for $\umu$ 
            \vspace{0,3cm}
      
      {\small
      $$\begin{array}{|c||c|c|c|c|c|c|c|c|c|c|c|c|c|c|c|c|c|c|c|c|c|}
      \hline
      \text{Component}  & 2 & 1 & 2 &  1 & 2 & 1 & 2 &  1 & 2 & 1 & 2 & 1 & 2 & 1 & 2 &  1 & 2 & 1 & 2 &  1 \\
      \hline
        \text{Content} & -3 & -3 & -2 & -2 &  -1 & -1 & 0 & 0 &  1 & 1 & 2 & 2 & 3 & 3 & 4 & 4 &  5 & 5 & 6 & 6  \\
        \hline 
         (6.1,2.2) & {\bf B_v} & {\bf B_v} & {\bf B_v} & A &  A & {\bf R} & B_h& A &  {\bf R} & B_h & {\bf B_v} & B_h & A & B_h & B_h & B_h &  B_h & {\bf R} & B_h & A   \\
          \hline 
         (5.2.1,3) & {\bf B_v} & {\bf B_v} & {\bf B_v} & A &  A & {\bf R} & B_h& A &  B_h & {\bf R} & {\bf R} & A & A & B_h & B_h & B_h &  B_h & {\bf R}& B_h & A   \\
                       \hline 
     \end{array}$$}

%
%
     \end{exa}

 \begin{Rem}\label{A2}
 The isomorphism between $V_{{\bf s}_1}$ and $V_{{\bf s}_2}$ implies the following fact. Assume that $(\mathfrak{j}_1,\ldots,\mathfrak{j}_n)$
  is a sequence of element in $I$ then  $f_{\mathfrak{j}_1}\ldots f_{\mathfrak{j}_n}.\uemptyset$ writes as a linear combinaison of the canonical basis elements:
 $$f_{\mathfrak{j}_1}\ldots f_{\mathfrak{j}_n}.\uemptyset=\sum_{\ulambda \in  \Uglov{e}{{\bf s}_1}}  a_{\lambda} G(\ulambda,{\bf s}_1)$$
 Then we have:
  $$f_{\mathfrak{j}_1}\ldots f_{\mathfrak{j}_n}.\uemptyset=\sum_{\ulambda \in  \Uglov{e}{{\bf s}_1}}  a_{\lambda} G(\Psi^e_{{\bf s}_1\to \sigma.{\bf s}_1} (\ulambda),{\bf s}_2)$$

 \end{Rem}
 
 In the case where ${\bf s}\in \mathcal{S}_e$, the Uglov bipartitions are then known as FLOTW bipartitions
 and they have a non recursive description given as follows:
 \begin{Prop}[Foda-Leclerc-Okado-Thibon-Welsh]\label{flotwdef}
Assume that ${\bf s}=(s_1,s_2)\in \mathcal{S}_e$.  The set $\Uglov{e}{\bf s}$ of Uglov bipartitions is the  set of bipartitions  $\ulambda={(\lambda^{1},\lambda^{2})}$ such that:
\begin{enumerate}
\item for all  $i\in \mathbb{Z}_{>0}$, we have:
\begin{align*}
&\lambda_i^{1}\geq{\lambda^{2}_{i+s_{2}-s_{1}}},\\
&\lambda^{2}_i\geq{\lambda^{1}_{i+e+s_1-s_{2}}};
\end{align*}
\item  for all  $k>0$, among the residues of the nodes of the vertical boundary of the form $(a,\lambda^c_a,c)$
 with $a\in \mathbb{Z}_{>0}$, $c\in \{1,2\}$ and $\lambda^c_a=k$, at least one element of  $\{0,1,...,e-1\}$ does not occur.
\end{enumerate}
\end{Prop}

 \subsection{Properties of Uglov bipartitions} 
 In the following, we will need several technical properties of Uglov bipartitions.  Let ${\bf s}\in \mathbb{Z}^2$
  and $\ulambda \in \Uglov{e}{{\bf s}}$. 
  
  \begin{abs}
  Assume that for  $\mathfrak{j}\in I$, we have a list of exactly $m$ normal $\mathfrak{j}$-nodes:
  $$\eta_1 <_{\bf s} \eta_2 <_{\bf s} \ldots <_{\bf s} \eta_m,$$
  then if follows from the definition of the crystal isomorphism (see \S \ref{bij})  that 
    for all ${\bf s}'\in \mathbb{Z}^2$ and $\Psi^e_{{\bf s}\to \sigma.{\bf s}'} (\ulambda)=:\umu $, we have a list 
  of exactly  $m$ normal $\mathfrak{j}$-nodes in $\umu$:
  $$\eta_1' <_{\bf s} \eta_2 '<_{\bf s} \ldots <_{\bf s} \eta_m'$$
For all $1\leq i\leq m$, one can thus canonically associate to the normal 
  $\mathfrak{j}$-node $\eta_i$ of  $\ulambda$ the normal   $\mathfrak{j}$-node $\eta_i$ of $\umu$. 
  \end{abs}
  
   \begin{abs}\label{per}
   Let $\gamma<_{\bf s} \gamma'$ be two removable $\mathfrak{j}$-nodes 
   with $\mathfrak{j}\in I$.  We say that $\gamma$ and $\gamma'$ are {\it $(1)$-connected } 
  if we have the following property:  if we remove $\gamma'$ from $\ulambda$ then 
   there exists 
   a set of nodes
$(\gamma_1,\ldots,\gamma_e)$ in the vertical boundary of the extended Young diagram of  the resulting bipartition $\ulambda'$, with 
 for all $i\in \{1,\ldots,e-1\}$, we have $\operatorname{cont}(\gamma_{i+1})=\operatorname{cont} (\gamma_i)+1$ and $c_{i+1}\geq c_i$.
  In this case,  we say that we have {\it period} in $\umu$ and we have  $\ulambda'\notin \Uglov{e}{{\bf s}}$ by \cite[Prop. 5.1]{JL}.
    In particular, it is easy to see that having a period for a bipartition satisfying $(1)$ in Def. \ref{flotwdef} is equivalent to 
     violate the condition $(2)$ in Def. \ref{flotwdef}.

Assume in addition that $\gamma$ and $\gamma'$ are normal $\mathfrak{j}$-nodes.  
 Let ${\bf s}'$ be in the orbit of ${\bf s}$ and take $\umu:=\Psi^e_{{\bf s}\to  {\bf s}'} (\ulambda)$. Then 
  $\gamma$ and $\gamma'$ correspond to two normal nodes 
   in $\umu$  which we denote by $\eta$ and $\eta'$. It follows from the 
    combinatorial description of the bijections that we cannot have 
     a $\mathfrak{j}$-node of the vertical boundary $\eta''$, consecutive to $\eta$ and  such that $\eta\prec_{\bf s} \eta''$. Indeed, in this case, one may check that we obtain 
 a period in $\umu$ which contradicts the fact that $\umu\in \Uglov{e}{{\bf s}'}$

%

 \end{abs}
 \begin{exa} The following example illustrates how the  bijection  $ \Psi^e_{{\bf s}\to  \sigma_1 {\bf s}}$   acts on the property of being $(1)$-connected. 
 is  
 
 Let ${\bf s}=(0,1)$, $e=3$ and consider $\ulambda =(3.2.2.1.1,3.3.1)$. 
 
  \vspace{1cm}  
  
  \centerline{
 \Bigg(\  \ \ \ \ytableausetup
{mathmode}
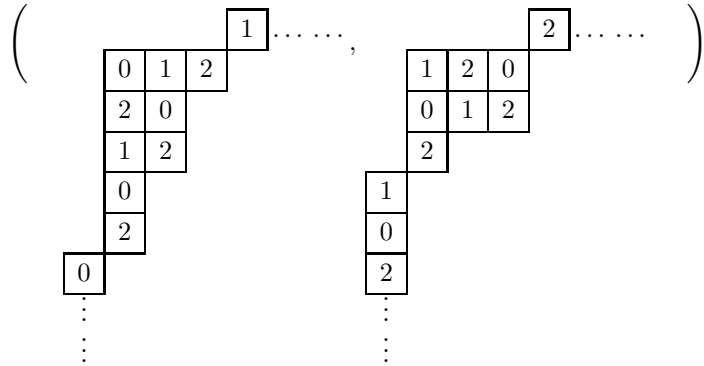
\begin{ytableau}
\none  &  \none  & \none  & \none & 1 & \none[\dots]& \none[\dots] \\
 \none &0 & 1   & 2 \\
\none & 2  & 0\\  
\none & 1  & 2\\
\none & 0 \\
\none & 2  \\
0  \\
\none[\vdots] \\
 \none[\vdots] \\
\end{ytableau},
\begin{ytableau}
\none  &  \none  & \none   & \none & 2 & \none[\dots]& \none[\dots] \\
\none & 1&  2& 0 \\
\none & 0 & 1  & 2 \\
\none & 2 \\
1  \\
0 \\
2  \\
\none[\vdots] \\
 \none[\vdots] \\
\end{ytableau}\  \ \ \  \Bigg)}
 We see that the node $(1,3,1)$ is $(1)$-connected  to $(3,2,1)$ which is itself $(1)$-connected  to $(5,1,1)$; and $(2,3,2)$  is $(1)$-connected  to $(3,2,1)$. 
 Now we have that $\Psi^3_{(0,1)\to (1,0)}(3.2.2.1.1,3.3.1)=(3.3.2.2.1.1,3.1)$. 
   \vspace{1cm}  
  
  \centerline{
 \Bigg(\  \ \ \ \ytableausetup
{mathmode}
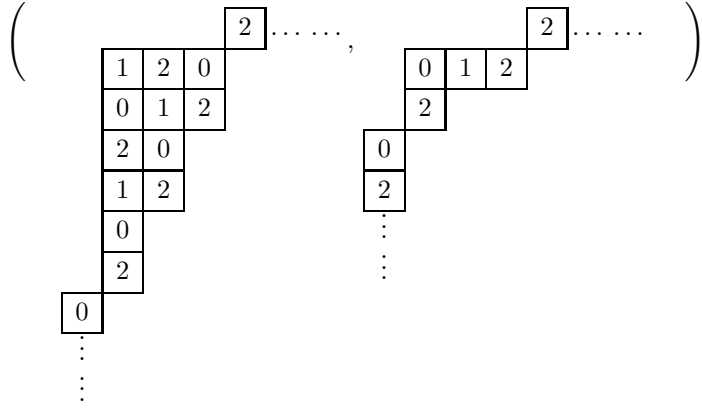
\begin{ytableau}
\none  &  \none  & \none  & \none & 2 & \none[\dots]& \none[\dots] \\
 \none &1 & 2   & 0 \\
\none & 0  & 1&2   \\
\none & 2  & 0\\
\none & 1  &2\\
\none & 0  \\
\none & 2  \\
0  \\
\none[\vdots] \\
 \none[\vdots] \\
\end{ytableau},
\begin{ytableau}
\none  &  \none  & \none   & \none & 2 & \none[\dots]& \none[\dots] \\
\none & 0&  1& 2 \\
\none & 2 \\
0 \\
2  \\
\none[\vdots] \\
 \none[\vdots] \\
\end{ytableau}\  \ \ \  \Bigg)}
  We see that the node $(2,3,1)$ is $(1)$-connected  to $(4,2,1)$ which is itself $(1)$-connected  to $(6,1,1)$; and $(1,3,2)$  is $(1)$-connected  to $(4,2,1)$. 
 
 \end{exa}

\subsection{Main result} 

The main result of this paper is:

\begin{Th}[Generalized Dipper-James-Murphy's conjecture]\label{djmc}
     Let $e\in \mathbb{Z}_{>1}\sqcup\{\infty\}$ and ${\bf s}\in \mathbb{Z}^2$ then $\ulambda \in\Uglov{e}{\bf s}(n)$    if and only if there exists  a sequence of residues    
$\mathfrak{i}_1,\ldots,\mathfrak{i}_n$ in $I$ and  
integers $c_{\ulambda,\umu}$ for $\umu\in \mathcal{P}^2 (n)$  such that:
$$f_{\mathfrak{i}_1}\ldots f_{\mathfrak{i}_n}.\uemptyset =c_{\ulambda,\ulambda}\ulambda+\sum_{\umu\prec_{\bf s} \ulambda}  c_{\ulambda,\umu} \umu.$$
with $c_{\ulambda,\ulambda}\neq 0$.
\end{Th}
See also section \ref{refo}, for a reformulation of this result.  
We will also see several consequences of this result. The aim of this rest of the paper is to prove this Theorem thanks to the previous preparatory materials  and to present two consequences.

\section{Proof of the main result}
We now prove Theorem \ref{djmc}.  The main point is to associate to each Uglov bipartition, a certain sequence of residues 
 which will play the role of the sequence $\mathfrak{i}_1,\ldots,\mathfrak{i}_n$ in $ I$ in the Theorem. 
 The important property about this sequence is that it will be an invariant   on the isomorphism  class of an   Uglov bipartition.

\subsection{Admissible residue sequence: FLOTW case}

In this subsection, we let ${\bf s}=(s_1,s_2)\in \mathcal{S}_e$ and $\ulambda \in\Uglov{e}{\bf s}$. 
First let us give another definition which will be necessary to define our sequence of nodes.  Let  $\gamma_1:=(a,b,c)$ be a removable $\mathfrak{j}$-node for some $\mathfrak{j}\in  I$. 
  \begin{itemize}
  \item If $c=1$ and   $\lambda_a^{1}={\lambda^{2}_{a+s_{2}-s_{1}}}$ then we set  $\gamma_2=(a+s_2-s_1,{\lambda^{2}_{a+s_{2}-s_{1}}},2)$. 
  \item If $c=2$ and $\lambda^{2}_a= {\lambda^{1}_{a+e+s_1-s_{2}}}$ then we set  $\gamma_2=(a+e+s_1-s_2,{\lambda^{1}_{a+e+s_{1}-s_{2}}},1)$.   
  \end{itemize}
then we say that $\gamma_1$ and $\gamma_2$ are {\it $(2)$-connected}.  
 If $\gamma$ and $\eta$ are two removable $\mathfrak{j}$-nodes then we set $\gamma \equiv  \eta$ if $\gamma$ and $\eta$ are $(1)$ or $(2)$ {\it connected } and we consider the transitive closure of this relation.

Now let $\ulambda=(\lambda^1,\lambda^2)$ be in $\Uglov{e}{\bf s}$.
Consider the maximal removable node $\gamma_1=(a,b,c)$  with respect to $<_{\bf s}$  and denote by  $\mathfrak{j}$ its residue. Note that by the definition of FLOTW $l$-partition, there cannot exist a node on the vertical boundary with the same residue greater than $\gamma_1$. 
  We now consider the sequence of  removable   $\mathfrak{j}$-nodes  
 given by all the removable $\mathfrak{j}$-nodes in the equivalence class. 
We write it  $(\gamma_1,\ldots,\gamma_r)$ (written in increasing order).

  \begin{Prop}\label{pf}
  Under the above hypotheses, we have the following properties:
  \begin{enumerate}
  \item for all $k=1,\ldots, r$ the bipartition obtained by removing the  $\mathfrak{j}$-nodes $(\gamma_k,\ldots,\gamma_r)$
   from $\ulambda$ is in $\Uglov{e}{\bf s}$. 
  \item  $\gamma_1$ is greater to any addable $\mathfrak{j}$-node of $\ulambda$ and any $\mathfrak{j}$-node of the vertical boundary of $\ulambda$
  \item if there exists a non virtual $\mathfrak{j}$-node of the horizontal boundary which is greater than $\gamma_1$ then $\gamma_1$ is $(1)$ {\it connected} with a node of the sequence 
   $(\gamma_2,\ldots,\gamma_{r-1})$
   \end{enumerate}
  \end{Prop}

  \begin{proof}
  Point $1$ just follows from the definition of FLOTW bipartitions.  Point $2$  has already been stated (it also follows from \cite[Lemma 4.2.5]{J} and \cite[Lemma 4.2.6]{J}). 
   Let us consider the last point. Assume that  there exists a  non virtual $\mathfrak{j}$-node of the horizontal boundary which is greater than $\gamma_1$ and take the smallest such node (note that such node cannot be greater than $\gamma_r$).  This situation occurs only in the case where this node is between  two nodes $\gamma_i$ and $\gamma_{j}$, with $j<i$  which are $(1)$-connected. From the properties of FLOTW bipartitions, we  see that 
    $\gamma_{j}$ must be $(1)$-connected with a node  $\gamma_{k}$, with $k<j$  etc. (if it is $(1)$-connected to its successive node then it must be 
      $(1)$-connected with another by the definition of FLOTW bipartitions) and  then the result follows.

\end{proof}
\begin{exa}
We illustrate the above proof with an example : let $e=3$, ${\bf s}=(0,2)$ and $\ulambda=(3.1.1,3.2.2.1.1)$ which is in $\Uglov{e}{\bf s}$. Write its extended Young diagram:

\vspace{1cm}

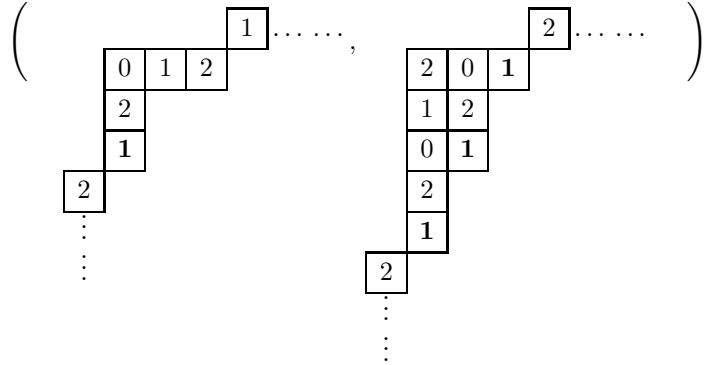
  \centerline{
 \Bigg(\  \ \ \ \ytableausetup
{mathmode}\begin{ytableau}
\none  &  \none  & \none  & \none & 1 & \none[\dots]& \none[\dots] \\
 \none &0 & 1   & 2 \\
\none & 2 \\
\none & {\bf 1}  \\
2  \\
\none[\vdots] \\
 \none[\vdots] \\
\end{ytableau},
\begin{ytableau}
\none  &  \none  & \none   & \none & 2 & \none[\dots]& \none[\dots] \\
\none & 2&  0& {\bf 1} \\
\none & 1&  2  \\
\none & 0&  {\bf 1}  \\
\none & 2 \\
\none & {\bf 1} \\
2 \\
\none[\vdots] \\
 \none[\vdots] \\
\end{ytableau}\  \ \ \  \Bigg)}

We here have $\gamma_1=(5,1,2)$, $\gamma_2=(3,1,1)$, $\gamma_3=(3,2,2)$, $\gamma_4=(1,3,2)$. 
We have a node $(1,2,1)$ of nature $B_h$ between  $\gamma_3$ and $\gamma_4$ and we see that $\gamma_1$ is $(1)$-connected with $\gamma_3$ which is itself
 $(1)$-connected to $\gamma_4$. 
\end{exa}

The admissible residue  sequence $\operatorname{Adm} (\ulambda)$ of $\ulambda$ is  then defined recursively as follows. Let 
$\ulambda'$ be the FLOTW bipartitions obtained after removing the $\mathfrak{j}$-nodes $(\gamma_1,\ldots,\gamma_r)$ from $\ulambda$ 
 $$\operatorname{Adm} (\ulambda)=\operatorname{Adm} (\ulambda'), \underbrace{\mathfrak{j},\ldots,\mathfrak{j}}_{r\text{ times}}$$

\subsection{Admissible residue sequence: general case}

We now explain how on can extend the notion of admissible residue sequence to the case of an arbitrary Uglov bipartition.

\begin{lemma}
Let ${\bf s}=(s_1,s_2) \in  \mathbb{Z}^2$ and and  let  $\ulambda$ be in $\Uglov{e}{\bf s}$.  For $\mathfrak{j}\in I$, denote by 
$$\gamma_1<_{\bf s} \gamma_2 <_{\bf s}  \ldots <_{\bf s}  \gamma_N$$
the normal $\mathfrak{j}$-nodes of $\ulambda$. Let $\sigma \in \widehat{\mathfrak{S}}_2$. Then 
  $\umu:=\Psi^e_{{\bf s}\to \sigma .{\bf s}} (\ulambda)$ admits exactlty
 $N$ normal $\mathfrak{j}$-nodes:
 $$\eta_1<_{\sigma.{\bf s}} \eta_2 <_{\sigma.{\bf s}}  \ldots <_{\sigma.{\bf s}}  \eta_N.$$
Assume that there exists   $1\leq m\leq N$ such that, for all $N\geq k\geq m$, the bipartition $\ulambda'$ obtained by removing  $\gamma_m$,   \ldots  $\gamma_N$
 from $\ulambda$ is in $\Uglov{e}{\bf s}$. Then the bipartition 
$\umu'$ obtained by removing  $\eta_m$,   \ldots   $\eta_N$
 from $\umu$ is in $\Uglov{e}{\sigma.{\bf s}}$ and we have $\umu'=\Psi^e_{{\bf s}\to \sigma .{\bf s}} (\ulambda')$
\end{lemma}
\begin{proof}
First note that it is sufficient to prove the lemma in the case where  $\sigma=\sigma_1$ (in the case where $s_1\leq s_2$) and $\sigma=\tau$.  The lemma is trivial in the case where $\sigma=\tau$ because then 
 $\umu=(\lambda^2,\lambda^1)$ and the bijection leave  the order on $\mathfrak{j}$-nodes invariant (see \cite[\S 6.1.1]{J}). Let us thus consider the case $\sigma=\sigma_1$ and $s_1\leq s_2$.  Taking the notations of the theorem,  we need to show that $\Psi^e_{{\bf s}\to \sigma .{\bf s}} (\ulambda')$
 is the bipartition obtained by removing the $N-m+1$  greatest normal  $\mathfrak{j}$-nodes of $\umu$.  To do this, we will essentially use the main result of \cite{JL}, which has been already mentioned and  which asserts that 
  the isomorphism $\Psi^{e}_{{\bf s}\to \sigma .{\bf s}}$ does not depend on $e$.  We can thus take $e=\infty$ to compute it. In this case, the residue of the node becomes its content and we thus have only two possible nodes with a given residue.  

Assume that the node $\gamma_m$ is on component $1$ and that its content is $j\in \mathbb{Z}$. By \cite[Prop. 4.1.1]{JL}, we have that $\ulambda \in  \Uglov{\infty}{\bf s}$ and  $\gamma_m$ is a good $j$-node for $e=\infty$. Moreover, we have  $\umu=\Psi^{\infty}_{{\bf s}\to \sigma .{\bf s}} (\ulambda)$ and the bipartition $\umu'$  obtained by removing the unique $j$-node from $\umu$ satisfies
 $\umu'=\Psi^{\infty}_{{\bf s}\to \sigma .{\bf s}} (\ulambda')$  where $\ulambda'$ is the bipartition obtained by removing $\gamma_m$ from $\ulambda$.  This follows from the fact that 
  $\Psi^{\infty}_{{\bf s}\to \sigma .{\bf s}}=\Psi^{e}_{{\bf s}\to \sigma .{\bf s}} $  is a crystal isomorphism. In addition, we have 
 $\umu'=\Psi^{e}_{{\bf s}\to \sigma .{\bf s}} (\ulambda')$   as this map does not depend on $e$. 
 
 If $\gamma_m$ is on component $2$ and if  its content is $j\in \mathbb{Z}$ and if there is no other removable node of content $j$, we conclude in the same manner. If we have two removable nodes $\gamma_m$ and $\gamma_{m+1}$ with content $j\in \mathbb{Z}$, we argue again in the same manner by removing these two nodes (which are 
 are successively good nodes in the case $e=\infty$). We conclude by induction. 

\end{proof}

Let ${\bf s}=(s_1,s_2) \in \mathbb{Z}^2$.  Let  $\umu$ be in $\Uglov{e}{\bf s}$. 
 Let ${\bf s}'$ be the element in  $\mathcal{S}_e$ 
 which is in the orbit of ${\bf s}$ modulo the action of $\widehat{\mathfrak{S}}_2$. Let $\ulambda:=\Psi^e_{{\bf s}\to {\bf s}'} (\umu)$.
 Consider the admissible residue sequence of $\ulambda$, then we define the admissible residue sequence of $\umu$ to be this residue sequence. 
We know that the normal $\mathfrak{j}$-nodes:
  $$\gamma_1<_{{\bf s}'} \gamma_2 <_{{\bf s}'}  \ldots <_{{\bf s}'}  \gamma_N$$
 associated to the admissible residue sequence of $\ulambda$ 
  are canonically associated to normal $\mathfrak{j}$-nodes: 
 $$\eta_1<_{\bf s} \eta_2 <_{\bf s}  \ldots <_{\bf s}  \eta_N$$ 
   of $\umu$.  By the result above,  if we remove these nodes from $\umu$ then we still have an Uglov bipartition which is 
    in the isomorphism class of the bipartition obtained by removing 
   $(\gamma_1, \gamma_2 ,  \ldots ,  \gamma_N)$ from $\ulambda$. 
   This sequence of nodes  have others interesting properties.
   \begin{Prop}\label{propb} Under the above hypotheses, 
\begin{enumerate}
\item there is no addable $\mathfrak{j}$-node in $\umu$ greater than $\eta_1$.  
\item  If there is a non virtual $\mathfrak{j}$-node of nature  $B_h$ greater than 
 $\eta_1$ in $\umu$  then there is no $\mathfrak{j}$-node of nature $B_v$ greater than $\eta_1$. 

 \end{enumerate}
   \end{Prop}
 \begin{proof}
 
 We argue by contradiction and assume first that there is a non virtual $\mathfrak{j}$-node of nature  $B_h$ and a $\mathfrak{j}$-node of nature $B_v$ greater than 
 $\eta_1$.  We have a sequence of bipartitions:
 $$\ulambda[1]:=\ulambda,\ulambda[2],\ldots,\ulambda[m]:=\umu$$
 and a sequence of elements in $\mathbb{Z}^2$:
  $${\bf s}[1]:={\bf s},{\bf s}[2],\ldots,{\bf s}[m]:={\bf s}'$$
where $s[j]:=\sigma. s[j-1]$ with $\sigma=\sigma_1$ or $\tau$ and  such that $\ulambda[j]=\Psi^{e}_{{\bf s}[j-1]\to {\bf s}[j]}  (\ulambda[j-1])$ for $j=2,\ldots,m$.
 Keeping the notation of this section, by the table in \S \ref{bij}, there is a non virtual $\mathfrak{j}$-node of nature $B_h$ greater than 
 $\gamma_1$. Thus by Proposition \ref{pf}, $\gamma_1$   must be $(1)$-connected with one of the other nodes $\gamma_j$
Now, by the definition of the vertical boundary, with the node of nature $B_v$ comes two nodes on the vertical boundary with residue $\mathfrak{j}$ and $\mathfrak{j}-1$. 
  By the discussion in \S \ref{per}, there exists $k\in \{1,\ldots,m\}$, such that in 
    $\ulambda[k] \in \Uglov{e}{{\bf s}[k]}$,  we have a $\mathfrak{j}$-node $\gamma_1'$ associated to $\gamma_1$  which is $(1)$-connected with another node and such that there exist a $\mathfrak{j}$-node of the vertical boundary $\eta$ which is consecutive to $\gamma_1'$ and such that $\gamma_1 ' <_{\bf s} \eta$. By the discussion in \S \ref{per}, this is impossible for an Uglov bipartition.

 The first point follows in fact from the second: if we have such an addable node, this means that there exists $j\in \{1,\ldots,m\}$ such that 
 in $\ulambda[j]$, the associated $\mathfrak{j}$-nodes   
 $$\eta_1'<_{\bf s} \eta_2 <_{\bf s}  \ldots <_{\bf s}  \eta_N'$$ 
are such that we have 
  a sequence of two consecutive nodes of nature, respectively $B_v$ and $B_h$ greater than $\eta_1'$ (see the table in \S \ref{bij}). This is thus impossible by the above discussion.
 \end{proof}

\begin{exa}
Take ${\bf s}=(0,1)$ and $e=3$. We consider the bipartition $\ulambda=(6.1,2.2)$ which is in $\Uglov{e}{\bf s}$. Then one can compute 
 $\Psi^e_{{\bf s}\to \sigma.{\bf s}} (\ulambda)$  for $\sigma\in \widehat{\mathfrak{S}}_2$ using the algorithm  in \cite{JU2} or one can use the program given in \cite{Jgap}. We get for all $k\in \mathbb{Z}$:
$$\Psi^e_{(0,1)\to (1+3k,0)} (\ulambda)=
\left\{ \begin{array}{rcl}
(5.2.1,3) & \text{ if } k\geq  0 \\
(2.2,6.1) & \text{ if } k=-1  \\
 (2.1,6.1.1) & \text{ if } k< -1 \\
\end{array}
\right.\ \ 
\Psi^e_{(0,1)\to (0,1+3k)} (\ulambda)=
\left\{ \begin{array}{rcl}
(3,5.2.1) & \text{ if } k> 1 \\
 (6.1.1,2.1) & \text{ if } k< 0
\end{array}
\right.
$$
We can compute the admissible sequence of  residues associated to $\ulambda$ and thus to all the bipartitions in its isomorphism class. It is given by:
$$1,0,2,2,1,1,2,0,1,1,2.$$

    \vspace{1cm}  
  
  \centerline{
 \Bigg(\  \ \ \ \ytableausetup
{mathmode}
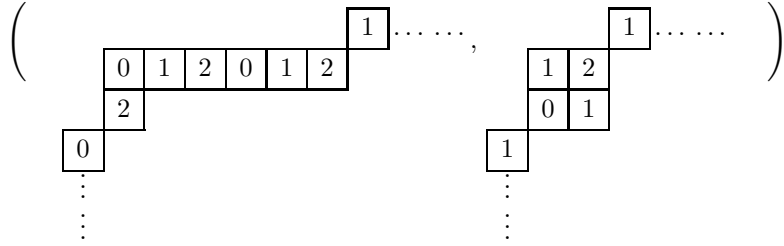
\begin{ytableau}
\none  &  \none  & \none  & \none  & \none & \none & \none &1 & \none[\dots]& \none[\dots] \\
 \none &0 & 1  & 2 & 0 & 1 & { 2}\\
\none & { 2} \\  
0  \\
\none[\vdots] \\
 \none[\vdots] \\
\end{ytableau},
\begin{ytableau}
\none  &  \none  & \none   & 1  & \none[\dots]& \none[\dots] \\
\none  &  1 & 2  \\
\none & 0  & 1\\  
1  \\
\none[\vdots] \\
 \none[\vdots] \\
\end{ytableau}\  \ \ \  \Bigg) 
}

\end{exa}

\subsection{The proof}
To prove our theorem,  we will show the following result.   Let $e\in \mathbb{Z}_{>1}\sqcup\{\infty\}$ and ${\bf s}\in \mathbb{Z}^2$. Let  $\ulambda \in\Uglov{e}{\bf s}(n)$   and denote $\operatorname{Adm} (\ulambda)=\mathfrak{i}_1,\ldots,\mathfrak{i}_n$. Then there exist   
integers $c_{\ulambda,\umu}$ for $\umu\in \mathcal{P}^2 (n)$  such that:
$$f_{\mathfrak{i}_n}\ldots f_{\mathfrak{i}_1}.\uemptyset =c_{\ulambda,\ulambda}\ulambda+\sum_{\umu\prec_{\bf s} \ulambda}  c_{\ulambda,\umu} \umu.$$
with $c_{\ulambda,\ulambda}\neq 0$.

Note that the admissible sequence of residue is by definition an invariant on the isomorphism class of an Uglov bipartition.  
it thus suffices to prove the following three properties: 
\begin{enumerate}
\item The theorem is true in the case where ${\bf s}=(s_1,s_2)\in \mathcal{S}_e$,
\item If the theorem is true for  all Uglov bipartitions $\ulambda\in \Uglov{\bf s}{e} (n)$ with 
${\bf s}=(s_1,s_2)\in \mathbb{Z}^2$ such that $s_1\leq s_2$ then it is true for all
 Uglov bipartitions $\ulambda\in \Uglov{\sigma_1.{\bf s}}{e} (n)$.
 \item If the theorem is true for  all Uglov bipartitions $\ulambda\in \Uglov{\bf s}{e} (n)$ with 
${\bf s}=(s_1,s_2)\in \mathbb{Z}^2$ such that $s_1\geq s_2$   then it is true for all
 Uglov bipartitions $\ulambda\in \Uglov{\tau.{\bf s}}{e} (n)$.
\end{enumerate}
Point $1$ is in fact a weak version of \cite[Lemma 5.7.20]{GJ} (the partial order $\ll_{\bf m}$  used in the book satisfies $\ulambda \ll_{\bf m} \umu \Rightarrow \ulambda \prec_{\bf s} \umu$). So we need to prove $2$ and $3$. To do this, 
assume that  ${\bf s}=(s_1,s_2)\in \mathbb{Z}^2$.  Let $n\in \mathbb{Z}_{>0}$.  One may assume that our result is true  for all Uglov bipartitions $\ulambda\in \Uglov{\bf s}{e} (n)$.  We show that the result is still true for   the Uglov bipartitions $\Psi^e_{{\bf s}\to \sigma_1.{\bf s}} (\ulambda)=:\ulambda^{\sigma} $ (in which case we assume $s_1\leq s_2$) and 
  $\Psi^e_{{\bf s}\to \tau.{\bf s}} (\ulambda)=:\ulambda^{\tau}$ (in which case we assume $s_1\geq s_2$). 

First assume that $\sigma\in \{\sigma_1,\tau\}$. 
  We  remove the 
 $a_k$ greatest removable $\mathfrak{j}_k$-nodes from   $\ulambda^{\sigma}$. Let 
 $\widetilde{\ulambda}^{\sigma}$ be the resulting  Uglov bipartition.  
 Assume that  $\umu^{\sigma}$  is the maximal element with respect to  $\prec_{\sigma.{\bf s}}$   appearing in 
 $f_{\mathfrak{i}_1}^{a_1}\ldots f_{\mathfrak{i}_{k}}^{a_{k}}.\uemptyset$ 
  and  that  $\umu^{\sigma} \neq \ulambda^{\sigma}$. The aim is to show a contradiction.

   By  induction, we have that  $\widetilde{\umu}^{\sigma}\prec_{\bf s} \widetilde{\ulambda}^{\sigma}$. 
    This implies that there exists $j\in \mathbb{Z}$ and a component $c\in \{1,2\}$ such that:
     \begin{itemize}
     \item The node of   $\widetilde{\ulambda}^{\sigma}$ with  content $j$ and component $c$ is of nature   $B_v$ or $B_h$,
     \item The node of   $\widetilde{\umu}^{\sigma}$  with  content $j$ and component $c$ is of nature $A$ and 
      the node of  $\umu^{\sigma}$ with   content $j$ and component $c$ is of nature $R$. 
     \item All the nodes in   ${\ulambda}^{\sigma}$  and $\umu^{\sigma}$   greater than  the one of content $j$ and component $c$
      have the same general nature.

     \end{itemize}

   \begin{abs}[{\bf Case $1$}]  The node of $\widetilde{\ulambda}^{\sigma}$  with  content $j$ and component $c$ is of  nature $B_h$.  
   
   \begin{itemize}
   \item Assume that $c=2$,  the table of natures  reads as follows:
      $$\begin{array}{|c||c|c|c|c|c|c|c|}
      \hline
      \text{Component} & \ldots  & 2& 1 & 2 &  1 & \ldots   \\
      \hline
        \text{Content} & \ldots & j-1 & j-1 &  j &  j & \ldots   \\
        \hline 
      \umu^{\sigma}  & \ldots & X_1 & X_2 &  R & X_3 & \ldots    \\
          \hline 
        \ulambda^{\sigma}  &    \ldots & Y_1 & Y_2 &  { B_h} &  Y_3   & \ldots \\
         \hline 
     \end{array}$$
As we have  $\widetilde{\umu}^{\sigma} \prec_{\sigma_1.{\bf s}}  \widetilde{\ulambda}^{\sigma}$  and because of 
 our assumptions, we must have the following table of natures:
        $$\begin{array}{|c||c|c|c|c|c|c|c|}
      \hline
      \text{Component} & \ldots  & 2& 1 & 2 &  1 & \ldots   \\
      \hline
        \text{Content} & \ldots & j-1 & j-1 &  j &  j & \ldots   \\
        \hline 
      \widetilde{\umu}^{\sigma}   & \ldots &  \left\{  \begin{array}{c} B_v \\ R \end{array}\right.  &  \left\{  \begin{array}{c} B_h \\ A \end{array}\right.  &  A &  \left\{  \begin{array}{c} B_h \\ R \end{array}\right.  & \ldots    \\
          \hline 
        \widetilde{\ulambda}^{\sigma}  &    \ldots &  \left\{  \begin{array}{c} B_h \\ A \end{array}\right.  &  \left\{  \begin{array}{c} B_v \\ R \end{array}\right.  &  { B_h} &   \left\{  \begin{array}{c} B_v \\ A \end{array}\right.   & \ldots \\
         \hline 
     \end{array}$$
     
     Assume that the  node of   $\widetilde{\ulambda}^{\sigma}$  with content $j$ in component $1$ is  of nature $A$ then the one 
      in $\widetilde{\umu}^{\sigma}$ is of nature $B_h$ and we have  $Y_3=R$ and $X_3=B_h$, we thus have   ${\umu}^{\sigma} \prec_{\sigma.{\bf s}} {\ulambda}^{\sigma}$.

    Assume that   the node of   $\widetilde{\ulambda}^{\sigma}$  with content $j$ in component $1$ is  of nature $B_v$.
    Then $Y_3=B_v$ and $X_3=R$. We immediately see that ${\umu}^{\sigma_1} \prec_{\sigma_1.{\bf s}} {\ulambda}^{\sigma_1}$.
   Let us consider the case $\sigma=\tau$. 
  By Proposition \ref{propb}, we have that the node of nature $B_h$ in component $2$ of $\widetilde{\ulambda}^{\tau}$ is virtual. Thus the node in component $1$ 
   and content $j-e$ of $\widetilde{\ulambda}$ is of nature $B_h$ and virtual.   This implies that  the node in component $1$ 
   and content $j-e$ of ${\ulambda}$ is of nature $B_h$ and virtual.

Note that we have  $s_2\geq s_1$,  so one may consider the map $\Psi^e_{\sigma_1 .{\bf s}\to {\bf s}}$ which has been explicitly described combinatorially (because then 
 $\sigma_1 .{\bf s}=:(s_2,s_1)$). 
   The algorithm for the computation of this  bijection  
 shows the following. If a  pair of consecutive nodes in components $2$ and $1$ are transformed into a pair of nodes in components $2$ and $1$ with nature 
  $X\in \{A,R,B_h,B_v\}$ and  $B_h$ (with $B_h$ virtual), this implies that the two nodes are of nature $B_h$ (with $B_h$ virtual). 
    We deduce that the node in component $1$ 
   and content $j-e$ of ${\ulambda}$ is also of nature $B_h$ and virtual. But then $Y_3\neq B_v$ and we get a contradiction.

   \item Assume that $c=1$,  the table of natures  reads as follows:
      $$\begin{array}{|c||c|c|c|c|c|c|c|}
      \hline
      \text{Component} & \ldots  & 2 & 1 & 2&  1 & \ldots   \\
      \hline
        \text{Content} & \ldots & j-1 & j-1  & j &  j & \ldots   \\
        \hline 
      \umu^{\sigma}   & \ldots & X_1 & X_2 &X_3 & R & \ldots    \\
          \hline 
        \ulambda^{\sigma}  &    \ldots & Y_1 & Y_2 & Y_3  & B_h & \ldots \\
         \hline 
     \end{array}$$
   If $X_3 \in \{B_v,R\}$ then  we immediately see that  that $\ulambda^{\sigma_1} \prec_{\sigma_1.{\bf s}} \umu^{\sigma}$. 
    If $X_3 \in \{A,B_h\}$  then if  $Y_3 \in \{A,B_h\}$, we again conclude in the same manner. We have to consider the case where 
     $Y_3 \in \{B_v,R\}$, we again conclude that $\ulambda^{\sigma_1} \prec_{\sigma_1.{\bf s}} \umu^{\sigma_1}$ (because of the table 
      giving the modification of the nature of nodes by isomorphism, the nature of the nodes of content $j$ in component 
       $1$ must be $R$ for $\umu^{\sigma_1}$ and $B_h$ for $\ulambda^{\sigma_1}$). Thus we have $\ulambda^{\sigma_1} <_{\sigma_1.{\bf s}} \umu^{\sigma_1}$.
      The result is clear for  ${\umu}^{\tau}$ and  ${\ulambda}^{\tau}$.
        
   \end{itemize}
   \end{abs}
   
   \begin{abs}[{\bf Case $2$}] The node of $\widetilde{\ulambda}^{\sigma}$  with  content $j$ and component $c$ is of  nature $B_v$.  
   \begin{itemize}
   \item Assume that $c=2$,  the node of content $j$ in component $2$ in $\umu$ is $R$ thus it implies that 
   the node of content $j-1$ in component $2$ in $\umu$ is $B_h$ or $A$. As  the node of content $j$ in component $2$ in $\ulambda$ is $B_v$ thus it implies that 
   the node of content $j-1$ in component $2$ in $\ulambda$ is $B_v$ or $R$. 
         $$\begin{array}{|c||c|c|c|c|c|c|c|}
      \hline
      \text{Component}  & \ldots  & 2& 1  & 2 &  1 & \ldots   \\
      \hline
        \text{Content} & \ldots & j-1 & j-1  & j &  j & \ldots   \\
        \hline 
      \umu^{\sigma}  & \ldots & \left\{  \begin{array}{c} B_h \\ A \end{array}\right. & X_1  & R & X_2  &  \ldots    \\
          \hline 
      \ulambda^{\sigma} &    \ldots & \left\{  \begin{array}{c} B_v \\ R \end{array}\right.  & Y_1 & B_v  & Y_2 & \ldots \\
         \hline 
     \end{array}$$

   But as $X_2$ and $Y_2$ have the same general nature, we must have $Y_1 \in \{ B_h, A\}$ 
    and $X_1 \in \{ B_v, R\}$. This implies that $Y_2 \in \{R, B_h\}$ 
     and $X_2 \in \{A,B_v\}$. But  $Y_2=R$  is impossible and $Y_2=B_h$ implies $X_2=A$ which contradicts the maximality of   $\umu^{\sigma}$.

   \item Assume that $c=1$,   again, we see that we have the following configuration:
   
            $$\begin{array}{|c||c|c|c|c|c|c|c|}
      \hline
      \text{Component} & \ldots  & 2 & 1 &  2 &  1 & \ldots   \\
      \hline
        \text{Content} & \ldots & j-1 & j-1 & j &  j & \ldots   \\
        \hline 
      \umu^{\sigma}   & \ldots & X_1 & \left\{  \begin{array}{c} B_h \\ A \end{array}\right.  &   X_2  &R &   \ldots    \\
          \hline 
         \ulambda^{\sigma} &    \ldots & Y_1& \left\{  \begin{array}{c} B_v \\ R \end{array}\right.   &Y_2&   B_v   & \ldots \\
         \hline 
     \end{array}$$
         We have that $X_2 \in \{ R, B_v\}$ and $Y_2\in \{A,B_h\}$. These two cases are not possible using the same reasoning as the case $1$ and $c=2$.       
 
      \end{itemize}

   \end{abs}
   \begin{Rem}
   Of course, one can ask if a similar result and proof can be obtained in the case of multipartitions. Most of the results presented in the last sections are still true but 
    one cannot argue as in this section to conclude. 
   
   \end{Rem}

 \subsection{Reformulation}\label{refo}
One can rephrase the main result as follows in terms of Young diagrams  (see \cite[\S 3.5.10]{GJ}).  To do this, let us introduce some more notations. Let 
 $\ulambda$ be a bipartition of $n$. 
A bijection $\mathfrak{s}: \mathcal{Y}(\ulambda) \to \{1,\ldots,n\}$ is called a  {\it $\ulambda$-tableau} and  we also say that $\ulambda$ is the shape of $\mathfrak{s}$. A
$\ulambda$-tableau  $\mathfrak{s}$  is called {\it row-standard}  if the sequence $\mathfrak{s}(a, 1,c)$,   $\mathfrak{s}(a, 2,c)$, $\ldots$  is strictly increasing for each $a$ and $c\in \{1,2\}$. 
The {\it residue sequence}  of a $\ulambda$-tableau $\mathfrak{s}$  is defined by $\eta_e (\mathfrak{s})=(\mathfrak{i}_1,\ldots,\mathfrak{i}_n)$ where $\mathfrak{i}_j$ is the residue of the node which is filled by the number $j$ in 
$\mathfrak{s}$. The theorem now becomes:
\begin{Cor}\label{djmco}
Let $e\in \mathbb{Z}_{>1}\sqcup\{\infty\}$ and ${\bf s}\in \mathbb{Z}^2$ then   $\ulambda$ is in $\Uglov{e}{\bf s}$ if and only if  for all bipartition $\umu$ admitting a row standard $\umu$-tableau $\mathfrak{s}$ such that $\eta_e (\mathfrak{s})=\operatorname{Adm} (\ulambda)$, we have $\umu\prec_{{\bf s}} \ulambda$. 
\end{Cor}
 
 \section{Consequences}
We quickly show how we obtain an application on the computation of canonical bases (this application is exactly the same as the one presented in \cite{J}). 
Let ${\bf s}:=(s_1,s_2)\in \mathbb{Z}^2$ and  $\ulambda \in  \Uglov{e}{\bf s}$, consider the associated admissible residue sequence:
$$\operatorname{Adm} (\ulambda)=\underbrace{\mathfrak{j}_1,\ldots,\mathfrak{j_{1}}}_{a_1},\ldots, 
\underbrace{\mathfrak{j}_k,\ldots,\mathfrak{j}_{k}}_{a_k},$$
where $\mathfrak{j}_{m}\in  I$ and $a_m\in \mathbb{Z}_{>0}$ for all 
 $m\in \{1,\ldots,k\}$ and where we assume that $\mathfrak{j}_m \neq \mathfrak{j}_{m+1}$ for all $s\in \{1,\ldots,k-1\}$. 
 We have now to work on the quantum group  $\mathcal{U}_v$ of  affine type $A$  (which can be seen as a deformation of $\mathfrak{g}_e$). It acts on the Fock space (where $v$ is an indeterminate). 
 With the same proof as \cite[Thm 6.4.2]{GJ}, we obtain that 
 $$f_{\mathfrak{j}_1}^{(a_1)}\ldots f_{\mathfrak{j}_k}^{(a_k)}\emptyset =\ulambda + 
 \sum_{\umu\prec_{\bf s} \ulambda}  c_{\ulambda,\umu} (v) \umu,$$
for Laurent polynomials  $c_{\ulambda,\umu} (v)$ and where  the $f_{\mathfrak{i}}^{(a)}$'s for $a\in \mathbb{Z}_{>0}$  stand for the divided powers of the Chevalley operators (see \cite{GJ}).   The specialization at $v=1$ of the above expression corresponds to the elements of the Theorem. As a consequence, 
as in \cite{J}, we obtain a LLT algorithm-like for the computation of the  canonical basis elements. 

One can also deduce from that the existence of basic sets for Hecke algebras of type $B_n$. We refer to \cite{GJ} for motivations and results around this theory.  
 Let $u$ be an indeterminate, let $V_1$ and $V_2$ be two indeterminates, let $R$ be a commutative ring with unit such that $\mathbb{Z} \subset R \subset \mathbb{C}$ and let $A:=R[u^{\pm 1},V_1,V_2]$.  Let $K$ be the field of fractions of $A$
let ${\bf H}_n$ be the Hecke algebra of type $B_n$ over $A$ with generators $\{T_i\ |\ i=0,\ldots,n-1\}$ and relations: 
\begin{itemize}
\item $(T_i-u)(T_i+1)=0$ for $i=1,\ldots,n-1$,
\item $(T_0-V_1)(T_0-V_2)=0$,
\item the type $B$ braid relations: $T_i T_j=T_j T_i$ with $|i-j|>1$, $T_i T_{i+1}T_i=T_{i+1} T_i T_{i+1}$ for $1\leq i\leq n-2$, and $T_0 T_1 T_0 T_1=T_1 T_0 T_1 T_0$. 
\end{itemize}
Let us denote  ${\bf H}_{K,n}= K\otimes_A {\bf H}_n $ then the simple ${\bf H}_{K,n}$-modules 
 are parametrized by the set of bipartitions. By Tits deformation theorem (see \cite[\S 1.2]{GJ}), we have:
$$\operatorname{Irr} ({\bf H}_{K,n})=\{ E^{\ulambda}\ |\ \ulambda \in \mathcal{P}^2 (n) \}$$
Let  $\theta: A \to L$ is a specialisation for a field $L$ and assume that:
$$\theta (u) \neq 1,\ \theta (V_1)=\theta (u)^{s_1}, \theta (V_2)=\theta (u)^{s_2}$$
Let $e>1$ be the order of $\theta (u) \in L^{\times}$. 
 The associated specalized  ${\bf H}_{L,n}= L\otimes_A {\bf H}_n $ is non semisimple in general and the representation theory is controled 
  by the decomposition matrix.

 For each $\ulambda \in  \Uglov{e}{\bf s}$, one can define the representation of ${\bf H}_{K,n}$:  
 $$P_{\ulambda}= E^\ulambda\oplus \bigoplus_{\umu\prec_{\bf s} \ulambda}  c_{\ulambda,\umu} E^\umu.$$
 The set of representations 
 $$\{ P_{\ulambda} \ |\ \ulambda \in  \Uglov{e}{\bf s}\}$$
satisfies the hypotheses of  \cite[Prop. 3.4.5]{GJ}. This implies that ${\bf H}_{k,n}$ 
 admits a basic set which is given by:
  $$\{ E_{\ulambda} \ |\ \ulambda \in  \Uglov{e}{\bf s}\}$$
 with respect to the partial order $\prec_{\bf s}$. 
 This means that the associated decomposition matrix is lower unitriangular and thus that $\Uglov{e}{\bf s} (n)$ is a natural parametrization set for the simple 
 ${\bf H}_{L,n}$-modules.  Note that this result is independant of the characteristic.  We refer to \cite{Bo} for an analogous result 
  obtained by a completely different process.

%
%

\vspace{0.5cm}
{\bf Address:}\\
\textsc{Nicolas Jacon}, Universit\'e de Reims Champagne-Ardenne, UFR Sciences exactes et naturelles, Laboratoire de Math\'ematiques EA 4535
Moulin de la Housse BP 1039, 51100 Reims, FRANCE\\  \emph{nicolas.jacon@univ-reims.fr}\\


\begin{thebibliography}{99}                                                                                               

\bibitem {A}\textsc{S.~Ariki},
 Representations of quantum algebras
and combinatorics of Young tableaux.
  Translated from the 2000 Japanese edition
and revised by the author. University Lecture Series, 26. American
Mathematical Society, Providence, RI, (2002).

\bibitem{AJ}
 \textsc{S. Ariki and N. Jacon},
Dipper-James-Murphy's conjecture for Hecke algebras of type $B$.
Representation Theory of Algebraic Groups and Quantum Groups Progress in Mathematics, 2011, Volume 284, 17-32.

\bibitem{AKT}
\textsc{A. Ariki, V. Kreiman, S. Tsuchioka},
 On the tensor product of two basic representations of $U_v (\widehat{sl}_e)$. 
 Advances in Mathematics, 
Volume 218, Issue 1, 1 May 2008, Pages 28-86.

\bibitem{Bo}
\textsc{C. Bowman,}
The many graded cellular bases of Hecke algebras. 
arXiv:1702.06579. 



\bibitem{CGG}
\textsc{M. Chlouveraki}, \textsc{I. Gordon} and \textsc{S. Griffeth},
 Cell modules and canonical basic sets for Hecke algebras from Cherednik algebras.
New trends in noncommutative algebra, 77-89, Contemp. Math., 562, Amer. Math. Soc., Providence, RI, 2012.


\bibitem{DJM}
\textsc{R. Dipper, G. D. James and E. Murphy},
Hecke Algebras of Type $B_n$ at Roots of Unity.
Proceedings of the London Mathematical Society s3-70(3):505-528.

\bibitem{GJb}
 \textsc{M. Geck and N. Jacon},
Canonical basic sets in type $B$. 
Journal of Algebra 306 (2006), 104-127.


\bibitem{GJ}
 \textsc{M. Geck and N. Jacon},
Representations of Hecke algebras at roots of unity. Algebra and Applications, 15. Springer-Verlag London, Ltd., London, 2011. 



\bibitem{GHJ}
 \textsc{T.Gerber, G. Hiss and N. Jacon}, 
 Harish-Chandra series in finite unitary groups and crystal graphs. 
International Mathematics Research Notices 22 (2015), 12206-12250. 


\bibitem{H1}
\textsc{J. Hu}, 
On a generalisation of the Dipper-James-Murphy conjecture,  
Journal of Combinatorial Theory Series A 118 (2011), no. 1, 78-93. 

\bibitem{H2}
\textsc{J. Hu, K. Zhou and D. Wang},
On the generalised Dipper-James-Murphy conjecture in quantum characteristic $2$,
Monatshefte f\"ur Mathematik 181(3), 2016. 


\bibitem{JU2}
 \textsc{N. Jacon},
Crystal graphs of irreducible $\mathcal{U}_v (\widehat{\mathfrak{sl}}_e)$-modules of level two and Uglov bipartitions. J. Algebraic Combin. 27 (2008), no. 2, 143-162

\bibitem{J}
 \textsc{N. Jacon},
Kleshchev multipartitions and extended Young diagram. Advances in Mathematics, 339, 1 December 2018, pages 367-403. 

\bibitem{Jgap}
 \textsc{N. Jacon},
GAP program ``crystal.g'' for the computation of crystal isomorphisms in affine type $A$.
 Available at \url{http://njacon.perso.math.cnrs.fr/}. 

\bibitem{JL}
 \textsc{N. Jacon and C. Lecouvey},
Crystal isomorphisms for irreducible highest weight $\mathcal{U}_q (\widehat{\mathfrak{sl}}_e)$-modules of higher level.
Algebras and Representation Theory 13, 2010, no. 4, 467-489.

\bibitem{LM}
 \textsc{B.Leclerc and H. Miyachi},
 Constructible characters and canonical bases.
 J. Algebra 277(2004), no. 1, p. 298-317.


\bibitem{Ug}
\textsc{D. Uglov},
Canonical Bases of Higher-Level $q$-Deformed Fock Spaces and Kazhdan-Lusztig Polynomials.
 Physical Combinatorics pp 249-299, Progress in Mathematics book series (PM, volume 191). 


 
\end{thebibliography}
\end{document}